\documentclass[12pt]{amsart}

\usepackage{color}
\definecolor{red}{rgb}{1.00,0.00,0.00}

\numberwithin{equation}{section}

\newtheorem{theorem}{Theorem}[section]
\newtheorem{lemma}[theorem]{Lemma}
\newtheorem{corollary}[theorem]{Corollary}
\newtheorem{proposition}[theorem]{Proposition}
\newtheorem{example}[theorem]{Example}

\newtheorem{remark}{Remark}
\newtheorem{definition}[theorem]{Definition}

\def\sA{\langle A\rangle}
\def\sB{\langle B\rangle}
\def\sC{\langle C\rangle}

\def\sS{\langle S\rangle}

\def\rA{k[A]}
\def\rB{k[B]}
\def\rC{k[C]}
 \def\rS{k[S]}
\def\C{\mathfrak{C}}
\def\N{\mathbb{N}}
\def\Q{\mathbb{Q}}
\def\Z{\mathbb{Z}}

\def\PP{\mathbb{P}}
\def\F{\mathbb{F}}
\def\a{{\bf a}}
\def\b{{\bf b}}
\def\c{{\bf c}}
\def\d{{\bf d}}

\def\u{{\bf u}}
\def\t{{\bf t}}
\def\xx{{\bf x}}
\def\yy{{\bf y}}
\def\iff{if and only if}
\def\rk#1{\hbox{\rm rank}\,(#1)}
\newcommand{\pd}{\mathrm {pd }}
\newcommand{\depth}{\mathrm {depth\, }}
\begin{document}
\title{On gluing semigroups in $\N^n$ and the consequences}

\author{Philippe Gimenez and Hema Srinivasan}
 \address{Instituto de Matem\'aticas de la Universidad de Valladolid (IMUVA),
Facultad de Ciencias, 47011 Valladolid, Spain.}
 \email{pgimenez@agt.uva.es}
 \thanks{The first author was partially supported by  PID2019-104844GB-I00,
{\it Ministerio de Ciencia e Innovaci\'on} (Spain).} 

 \address{Mathematics Department, University of
Missouri, Columbia, MO 65211, USA.}
 \email{SrinivasanH@missouri.edu}

\thanks{{\bf Keywords}: semigroup rings, gluing, degenerate semigroups, Cohen-Macaulay rings. }

\begin{abstract} 
A semigroup $\sC$ in $\N^n$ is a gluing of $\sA$ and $\sB$ if its finite set of generators $C$ splits into two parts, $C=k_1A\sqcup k_2B$ 
with $k_1,k_2\geq 1$, and  the defining ideals of the  corresponding semigroup rings satisfy that  $I_C$ is generated by $I_A+I_B$ and one extra element.
Two semigroups $\sA$ and $\sB$  can be glued if there exist positive integers $k_1,k_2$ such that, for $C=k_1A\sqcup k_2B$, $\sC$ is
a gluing of $\sA$ and $\sB$.
Although any two numerical semigroups, namely semigroups in dimension $n=1$,  can always be glued, it is no longer the case in higher dimensions.  In this paper, we give  necessary and sufficient conditions on $A$ and $B$ for the existence of a gluing of $\sA$ and $\sB$, and give examples to illustrate why they are necessary.  These generalize and explain the previous known results on existence of gluing.   We also prove that the glued semigroup $\sC$  inherits the properties like Gorenstein or Cohen-Macaulay from the two parts $\sA$ and $\sB$.
\end{abstract}

\maketitle

\begin{flushright}
{\em Dedicated to Professor J\"urgen Herzog on his 80th Birthday.}
\end{flushright}
\null\vskip 1cm
 
\section*{Introduction}

Given a finite subset of $\N^n$, $A=\{\a_1,\ldots,\a_p\}$. let $\sA$  denote the semigroups generated
 by $A$ over an arbitrary field $k$.  The  $n\times p$ matrix of natural numbers, whose columns 
are $\a_1,\ldots,\a_p$ will also be denoted by $A$. The semigroup rings $\rA$ is  the $k$-algebra generated by the monomials in
$k[\t]=k[t_1,\ldots,t_n]$ associated to the elements in $A$:  $\rA=k[\t^{\a_j},\ 1\leq j\leq p]$. 
If we consider the $k$-algebra homomorphism $\phi_A: k[x_1, \ldots, x_p]\to k[t_1, \ldots, t_n]$ given by $\phi_A(x_j) =\t^{\a_j} =\prod_{i=1}^n t_i^{a_{ij}}$, one has that $\rA\simeq k[x_1, \ldots, x_p]/I_A$ where $I_A$ is the kernel of $\phi_A$.   The ideal $I_A$ is a binomial ideal which is prime.  Further,  $\dim{\rA}=\rk{A}$ (\cite[Lem. 4.2]{sturm}).  The semigroup $\sA$ is said to be degenerate if $\rk{A}<n$.  The semigroup $\sA$ is of embedding dimension $p$ if the elements in $A$ minimally generate $\sA$.

\smallskip
Let $A=\{\a_1,\ldots,\a_p\}$ and $B=\{\b_1,\ldots,\b_q\}$ be two disjoint sets in $\N^n$ minimally generating $\sA$ and $\sB$ respectively.  Let $\rA\simeq k[x_1, \ldots, x_p]/I_A$ and $\rB\simeq k[y_1, \ldots, y_q]/I_B$ be the corresponding semigroup rings. 
Take two positive integers $k_1$ and $k_2$, and consider the set $C=k_1A\sqcup k_2B$ and the semigroup $\sC$. 
Setting $R=k[x_1,\ldots,x_p,y_1,\ldots,y_q]$, one has that $\rC\simeq R/I_C$.
We say that $\sC$ is a {\it gluing} of $\sA$ and $\sB$ if $I_C$ is generated by $I_A\cdot R+I_B\cdot R$ and one extra element $\rho\in R$ of the form
$\rho=\xx^\c-\yy^\d$ for some $\c\in\N^p$ and $\d\in\N^q$.
When this occurs, we will denote by $C = k_1A\Join k_2B$.
And we say that $\sA$ and $\sB$ {\it can be glued} if there exists $k_1$ and $k_2$ such that $C = k_1A\Join k_2B$.

\smallskip

The case when $n=1$ is well-understood: in this case one can always assume w.l.g. that $\gcd{(a_1,\ldots,a_p)}=1$ and $\gcd{(b_1,\ldots,b_q)}=1$,
i.e., that $\sA$ and $\sB$ are numerical semigroups.  Two numerical semigroups can always be glued: for any $k_1\in\sB$ and $k_2\in\sA$,
one has that $C = k_1A\Join k_2B$.
The case where one of the two semigroup rings has rank 1 is also completely charaterized in \cite[Thms. 3 and 5]{sforum20} where we say when and how two semigroups
can be glued when $\dim\rB=1$.
These results will be generalized  in this paper where we give necessary and sufficient conditions on $A$ and $B$ 
when the two semigroups $\sA$ and $\sB$ can be glued.  
The results explain the conditions in \cite[Thms 3 and 5]{sforum20}. 

\smallskip

The concept of gluing has its origin in the classical papers by Herzog \cite{He} and Delorme \cite{De} where the structure of complete intersection
numerical semigroups was established.  In \cite{He} Herzog completely characterized the semigroups generated by three positive integers and in \cite{De} Delorme extended the characterization to complete intersections  minimally generated by $n$ positive integers. The concept was later formally defined by Rosales in \cite{rosales}
in order to study complete intersection semigroups in $\N^n$. A characterization of complete intersection semigroups in terms of gluing is
obtained in \cite{RG} for simplicial semigroups, and finally generalized in \cite{FMS} for arbitrary semigroups in $\N^n$: 
A semigroup is a complete intersection
if and only if it is a gluing on two complete intersections of smaller embedding dimension.
In this paper, we show in Theorem \ref{thmCM} that the gluing of two semigroups is a complete intersection, respectively Gorenstein or Cohen-Macaulay  if and only if
both parts have the same properties. 

\smallskip

The structure of the paper is as follows. In Section \ref{secDimDepth} we use the results previously obtained in \cite{jpaa19} on the structure of
the minimal graded free resolution of semigroups obtained by gluing to express the projective dimension of $\rC$ in terms of the projective dimensions of
$\rA$ and $\rB$ when $C = k_1A\Join k_2B$. We then prove that the same relation holds between the dimension of the three rings and deduce
the expected characterization for Cohen-Macaulay, Gorenstein and complete intersection glued semigroups (Theorem \ref{thmCM}).
In Section \ref{secWhenGluing}, we give two conditions on $A$ and $B$ that must hold when $\sA$ and $\sB$ can be glued. We  then
give sufficient conditions, very close to the previous necessary ones, that imply that $\sA$ and $\sB$ can be glued (Theorem \ref{thmHowToGlue}). Using this, 
we recover in Section \ref{secConsequences} the already known results on numerical semigroups and when $\dim\rB=1$.  We consider the homogeneous case and show how to glue in $\N^3$ two nondegenerate homogeneous semigroups in $\N^2$ setting them as degenerate semigroups in $\N^3$.  We give examples to illustrate many of our results.

\section{Dimension, depth and Cohen-Macaylayness}\label{secDimDepth}

Consider finite subsets $A=\{\a_1,\ldots,\a_p\}$, $B=\{\b_1,\ldots,\b_q\}$ and $C=k_1A\sqcup k_2B$ of $\N^n$ where
$k_1,k_2\geq 1$ are positive integers.  Let  $\sA$, $\sB$ and $\sC$ be the semigroups generated by these sets and $\rA\simeq k[x_1, \ldots, x_p]/I_A$, $\rB\simeq k[y_1, \ldots, y_q]/I_B$ and $\rC\simeq k[x_1, \ldots, x_p, y_1, \ldots, y_q]/I_C$ be the corresponding semigroup rings.  Let $R=  k[x_1, \ldots, x_p, y_1, \ldots, y_q]$.  
\begin{definition}\label{gluing} $\sC$ is said to be a  gluing of $\sA$ and $\sB$, written as $C = k_1A\Join k_2B$, if $I_C=I_A\cdot R+I_B\cdot R+\langle\rho\rangle$ with
$\rho=\xx^\c-\yy^\d$ for some $\c\in\N^p$ and $\d\in\N^q$. 
The binomial $\rho$ will be called a {\it gluing binomial}.  
\end{definition}
Note that if $C$ is degenerate, i.e., $\dim\rC=\rk C<n$, then we can embed all three  semigroups in a smaller dimension with non degenerate $C'$,  $A'$ and $B'$ with isomorphic semigroup rings, and $C'= k_1A'\Join k_2B'$.   Thus, we can always assume that {\bf $C$ is non degenerate} without loss of generality, i.e.,
$\rk{A|B}=n$. 

\smallskip
When $C = k_1A\Join k_2B$, a minimal graded free resolution of $\rC$ can be obtained from the minimal graded free resolutions of $\rA$ and $\rB$ as follows:

\begin{theorem}[{\cite[Thm. 6.1 \& Cor. 6.2]{jpaa19}}]\label{thmJPAA}
Suppose that $C= k_1A\Join k_2B$. Let $\rho$ be a gluing binomial, and $\F_A$ and $\F_B$ be minimal graded free resolutions of $\rA$ and $\rB$ respectively.
Then:

\begin{enumerate}
\item
$F_A\otimes F_B$ is a  minimal graded free resolution of $R/I_A\cdot R+I_B\cdot R$. 
\item
A  minimal graded free resolution of $\rC$ can be obtained as the mapping cone
of $\rho: \F_A\otimes \F_B\to \F_A\otimes \F_B$ where $\rho$ is induced by multiplication by $\rho$.  
\item
The Betti numbers of $\rC$ are, for all $i\geq 0$:
$$
\beta_i(\rC) =
\sum _{i'=0}^i \beta_{i'}(\rA)[\beta_{i-i'}(\rB)+\beta_{i-i'-1}(\rB)]=
\sum _{i'=0}^i \beta_{i'}(\rB)[\beta_{i-i'}(\rA)+\beta_{i-i'-1}(\rA)]\,.
$$  
\item\label{itPD}
$\pd(\rC)=\pd(\rA)+\pd(\rB)+1$.
\end{enumerate}
\end{theorem}

The dimension and the depth of the ring of a semigroup obtained by gluing are given in the following:

\begin{theorem}\label{thmDimDepth}
If $C = k_1A\Join k_2B$, then:
\begin{enumerate}
\item\label{itDepth}
$\depth\,\rC=\depth\,\rA+\depth\,\rB-1$;
\item\label{itDim}
$\dim\rC=\dim\rA+\dim\rB-1$.
\end{enumerate}
\end{theorem}

\begin{proof}
(\ref{itDepth}) is a direct consequence of Theorem \ref{thmJPAA} (\ref{itPD}) and the Auslander-Buchsbaum formula: on one hand, $\pd{(\rC)}=\pd{(\rA)}+\pd{(\rB)}+1$, on the other
$\pd{(\rA)}=p-\depth\,\rA$, $\pd{(\rB)}=q-\depth\,\rB$ and $\pd{(\rC)}=p+q-\depth\,\rC$, and (\ref{itDepth}) follows.

\smallskip 

By definition,  $\rA\simeq k[x_1,\ldots,x_p]/I_A$, $\rB\simeq k[y_1,\ldots,y_q]/I_B$ and
$\rC\simeq R/I_C$ for $R=k[x_1,\ldots,x_p,y_1,\ldots,y_q]$ and 
$I_C=I_A\cdot R+ I_B\cdot R +\langle\rho\rangle$ for some binomial $\rho=\underbar{x}^\c-\underbar{y}^\d$ 
with $\c\in\N^p$ and $\d\in\N^q$.  Set $s=\dim\rA$ and $t=\dim\rB$.
As $I_A\cdot R$ and $I_B\cdot R$ are ideals in $R$ minimally generated by elements involving disjoint sets of variables,
$\dim R/(I_A\cdot R+I_B\cdot R)=\dim k[x_1,\ldots,x_p]/I_A+\dim k[y_1,\ldots,y_q]/I_B=\dim\rA+\dim\rB = s+t$.  Thus, for  $I_C=I_A\cdot R+ I_B\cdot R +\langle\rho\rangle$, we get

$$s+t-1\leq \dim \rC=\dim R/I_C\leq s+t$$
If $\dim \rC= s+t$, then $\rk A=s$, $\rk B=t$ and $\rk C=s+t$. 
So we can eventually rearrange the elements in $A$ and $B$ so that 
$\{\a_1,\ldots,\a_s,\b_1,\ldots,\b_t\}$ are $\Q$-linearly independent and write
$$
\forall i,\, s+1\leq i\leq p,\ \a_i=\sum_{k=1}^{s}a_{ki}\a_k,\ \hbox{ and }\ 
\forall i,\, t+1\leq i\leq q,\ \b_i= \sum_{k=1}^{t}b_{ki}\b_k,
$$
for some $a_{ki},b_{ki}\in\Q$.
Since $I_C$ is a binomial prime ideal, it is minimally generated by binomials that are the difference of two monomials of disjoint support.
Consider such a binomial $w$ in $I_C$.
Given an element $\gamma\in\Z^r$, denote 
$\gamma_+=\{j,\ 1\leq j\leq r\,/\ \gamma_j>0\}$ and $\gamma_-=\{j,\ 1\leq j\leq r\,/\ \gamma_j<0\}$.
Then, $w$ is of the form
$$
w=\prod_{j\in \alpha_+}x_j^{\alpha_j} \prod_{j\in \beta_-} y_j^{-\beta_j}-
           \prod_{j\in \alpha_-}x_j^{-\alpha_j}\prod_{j\in \beta_+}y_j^{\beta_j}
$$
for some $\alpha\in\Z^p$ and $\beta\in\Z^q$.
Then we have 
$\sum_{j=1}^{p}\alpha_j\a_j=\sum_{j=1}^{q}\beta_j\b_j$.
So we get
$$
(\alpha_1+\sum_{i=s+1}^{p}\alpha_i a_{1i} )\a_1+\cdots+(\alpha_s+\sum_{i=s+1}^{p}\alpha_i a_{si} )\a_s
-
(\beta_1+\sum_{i=t+1}^{q}\beta_i b_{1i} )\b_1-\cdots-(\beta_t+\sum_{i=t+1}^{q}\beta_i b_{ti} )\b_t=0
$$
and since these vectors are $\Q$-linearly independent, we have that the coefficients are all zero. Hence, we get that
$$
\sum_{i=1}^{p}\alpha_i\a_i=(\alpha_1+\sum_{i=s+1}^{p}\alpha_i a_{1i} )\a_1+\cdots+(\alpha_s+\sum_{i=s+1}^{p}\alpha_i a_{si} )\a_s=0,
$$
so for $f=\prod_{i\in \alpha_+}x_i^{\alpha_i}-\prod_{i\in \alpha_-}x_i^{-\alpha_i}$, one has that $f\in I_C\cap k[x_1,\ldots,x_p]=I_A$.
Similarly, $g=\prod_{i\in \beta_+} y_i^{\beta_i}-\prod_{i\in \beta_-}y_i^{-\beta_i}\in I_B$.
Then $w=(\prod_{i\in \beta_+} y_i^{\beta_i}) f+(\prod_{i\in \alpha_-}x_i^{-\alpha_i}) g\in I_A\cdot R+I_B\cdot R$.
We have shown that if $\dim \rC= \dim\rA+\dim\rB$, then $I_C= I_A\cdot R+I_B\cdot R$, a contradiction.
Thus, $\dim\rC=\dim\rA+\dim\rB-1$.
\end{proof}

A direct consequence is the following result proved in \cite[Thm. 2]{sforum20} assuming that $\rA$ and $\rB$ are Cohen-Macaulay. This hypothesis can now be removed.

\begin{theorem}\label{thmDeg}
If $n\ge 2$, then $\sA$ and $\sB$ can not be glued unless at least one of $\sA$ and $\sB$ is degenerate. 
\end{theorem}

\begin{proof}
If $\sC$ is a gluing of $\sA$ and $\sB$, by Theorem \ref{thmDimDepth} (\ref{itDim}), 
$\dim\rA+\dim\rB=\dim\rC+1\leq n+1$.  Hence if $\sA$ and $\sB$ are nondegenerate, 
then we get $2n\le n+1$ which is possible if and only if $n=1$. So if $n\geq 2$, we cannot glue $\sA$ and $\sB$ unless one of them is degenerate. 
\end{proof}

As a consequence of Theorem \ref{thmDimDepth}, one also gets a characterization of the Cohen-Macaulay, Gorenstein and complete intersection properties for a gluing:

\begin{theorem}\label{thmCM}
If $C = k_1A\Join k_2B$, then:
\begin{enumerate}
\item\label{itCM}
$\rC$ is Cohen-Macaulay if and only if $\rA$ and $\rB$ are.
\item\label{itGor}
$\rC$ is Gorenstein if and only if $\rA$ and $\rB$ are.
\item\label{itCI}
$\rC$ is a complete intersection if and only if $\rA$ and $\rB$ are.
\end{enumerate}
\end{theorem}

\begin{proof}
We have the three inequalities
$\depth\rC\,\leq\dim\rC$, $\depth\rA\,\leq\dim\rA$ and $\depth\rB\,\leq\dim\rB$, and 
equality holds in the first one if and only if it holds in the other two by Theorem \ref{thmDimDepth} (\ref{itDepth}) and (\ref{itDim}). This proves (\ref{itCM}).

Then (\ref{itGor}) is a consequence of (\ref{itCM}) and the fact that if $\rA$, $\rB$ and $\rC$ are Cohen-Macaulay, then the Cohen-Macaulay type of  $\rC$ is the product of the types of $\rA$ and  $\rB$ by \cite[Thm. 6.1 (4)]{jpaa19} so it is equal to 1 if and only if both types are equal to 1.

(\ref{itCI}) is a consequence of Theorem \ref{thmDimDepth} (\ref{itDim}) and the fact that, by definition of gluing, the number of minimal generators of $I_C$, $\mu(I_C)$,  is $\mu(I_A)+\mu(I_B)+1$. Indeed, $\mu(I_A)\geq p-\dim{\rA}$ and $\mu(I_B)\geq q-\dim{\rB}$, and hence $\mu(I_C)\geq p+q-\dim{\rA}-\dim{\rB}+1=n-(\dim{\rA}-\dim{\rB}-1)=n-\dim{\rC}$, and equality holds if and only if $\mu(I_A)=p-\dim{\rA}$ and $\mu(I_B)= q-\dim{\rB}$.
\end{proof}

\begin{remark}{\rm
As we mentioned in the introduction, for complete intersections a stronger result is known \cite{FMS}: Namely, 
a semigroup in $\N^n$ is a complete intersection if and only if it is a gluing on two complete intersections of smaller embedding dimension.
This stronger result is no longer valid for Gorenstein or Cohen-Macaulay semigroups since there exist, even in the case of numerical semigroups,
examples of Gorenstein semigroups that are not a gluing of semigroups of smaller embedding dimension.
For example, by \cite[Prop. 5.2]{jpaa19}, one has that a numerical semigroup $\sC$ minimally generated by an arithmetic sequence of length $n\geq 3$, $C=\{a,a+d,\ldots,a+(n-1)d\}$, is a gluing if and only if $n=3$, $a$ is even, and $d$ is odd and relatively prime to $a/2$. Note that these are also the only numerical semigroup minimally generated by an arithmetic sequence that are a complete intersection.
In fact, a numerical semigroup minimally generated by an arithmetic sequence of length $n\geq 4$ is never a gluing of two smaller numerical semigroups. Thus, by \cite{ja13},
for all $n\geq 4$ and $1\leq t\leq n-1$, there is a numerical semigroup of embedding dimension $n$ and Cohen-Macaulay type $t$ which is not a gluing of two smaller numerical semigroups.
}\end{remark}

Theorem \ref{thmDeg} states that if $n\geq 2$ and $\sC$ is a gluing of $\sA$ and $\sB$, then at least one of the semigroups $\sA$ and $\sB$ has to be degenerate. But one can be more precise: if $\sC$ is a nondegenerate gluing of $\sA$ and $\sB$ then, by Theorem \ref{thmDimDepth} (\ref{itDim}),
\begin{equation}\label{eq1}
\dim{\rB}=n-\dim{\rA}+1\,.
\end{equation}
In particular, if $\sA$ is nondegenerate, then $\rB$ must have dimension 1 which is the most degenerate case.  

\begin{corollary}\label{corNondeg}
Given $\sA$ and $\sB$ two semigroups in $\N^n$, if $\sA$ and $\sB$ can be glued into a nondegenerate semigroup $\sC$, then $\rk A+\rk B=n+1$. 
In particular,  if $\sA$ and $\sB$ can be glued and  one of them, say, $\sA$ is nondegenerate, then $\rk B=1$.
\end{corollary}

If $\dim{\rA}$ or $\dim{\rB}=1$, the semigroups $\sA$ and $\sB$ that can be glued are completely characterized by \cite[Thm. 3 and 5]{sforum20}.
If $n=2$, this is the only case and \cite[Thm. 6]{sforum20} can be improved using Theorems \ref{thmDeg} and \ref{thmCM} to obtain
a complete characterization of the semigroups in $\N^2$ that can be glued:

\begin{theorem}\label{thmDim2}
Let $A=\{\a_1,\ldots,\a_p\}$ and $B=\{\b_1,\ldots,\b_q\}$ be two finite subsets of $\N^2$. 
Then, $\sA$ and $\sB$ can be glued \iff{}
one of the two subsets, for example $B$, satisfies one of the following conditions:
\begin{itemize}
\item
either $B$ has one single element $\b$, i.e., $q=1$,
\item
or $q\geq 2$ and $B = \b \begin{bmatrix}u_1 & \ldots & u_q\end{bmatrix}$ for
some $u_1,\ldots,u_q\in\N$,
\end{itemize}
and,  moreover, the system $A\cdot x=d\b$ has a solution in $\N^2$ for some $d\in\N$.
\end{theorem}

Assuming now that $n\geq 3$ and that $\sA$ and $\sB$ are both degenerate, we wonder when $\sA$ and $\sB$ can be glued into a nondegenerate semigroup $\sC$. 
One can assume without loss of generality that $\dim{\rA}\geq\dim{\rB}$ (if not, one can switch them).
By (\ref{eq1}), if $2\leq \dim{\rA}\leq n-1$, then $2\leq \dim{\rB}\leq n-1$:
\begin{itemize}
\item
If $n=3$, one could have $\dim{\rA}=\dim{\rB}=2$,
\item
if $n=4$, one could have  $\dim{\rA}=3$ and $\dim{\rB}=2$,
\item
if $n=5$, one could have $\dim{\rA}=\dim{\rB}=3$ or $\dim{\rA}=4$ and $\dim{\rB}=2$,
and so on. 
\end{itemize}

\smallskip
When $n=3$, $\sA$ and $\sB$ could be both degenerate of dimension 2 and glued into a nondegenerate semigroup $\sC$ as \cite[Ex. 7]{jpaa19} shows.
In this example, we glue the twisted cubic with itself in $\N^3$. Let's recall this in Example \ref{exTwistedYes} and then show in Example \ref{exTwistedNo} 
that two degenerate semigroups in $\N^3$
whose defining ideal is the defining ideal of the twisted cubic could also not be glued.
We will later study in Section \ref{secHowToDim2} how to glue in $\N^3$ two nondegenerate homogeneous semigroups in $\N^2$ setting them as degenerate semigroups in $\N^3$ just as we do in the next example.

\begin{example}\label{exTwistedYes}{\rm
As shown in \cite[Ex. 7]{jpaa19} and \cite[Ex. 1]{sforum20}, if one considers the two degenerate semigroups $\sA$ and $\sB$ in $\N^3$ generated by the columns of the integer
matrices $A=\begin{pmatrix} 4&3&2&1\\ 0&1&2&3\\ 0&0&0&0\end{pmatrix}$ and $B=\begin{pmatrix} 3&3&3&3\\ 3&2&1&0\\ 0&1&2&3\end{pmatrix}$,
their defining ideal 
$I_A\subset k[x_1,\ldots,x_4]$ and $I_B\subset k[y_1,\ldots,y_4]$ 
are both the defining ideal of the twisted cubic. In other words, $\rA\simeq \rB\simeq \rS$  where 
$\sS$ is the semigroup in $\N^2$ generated by $S=\{(3,0), (2,1), (1,2), (0,3)\}$.
One can check that for $C=A\cup B$, one has that $I_C=I_A+I_B+\langle y_1^2-x_1x_4^2\rangle$ and hence, $C =A\Join B$.
}\end{example}

\begin{example}\label{exTwistedNo}{\rm
Now if one substitutes 
$B$ by $B'=\begin{pmatrix} 2&2&2&2\\ 3&2&1&0\\ 1&2&3&4\end{pmatrix}$
in Example \ref{exTwistedYes}, one also has that
$I_{B'}$ is the defining ideal of the twisted cubic but for $C'=A\cup B'$, $\langle C'\rangle$ is not a gluing of $\sA$ and $\langle B'\rangle$ since 
$I_{C'}=I_A+I_B+\langle y_1^2y_2-x_3x_4^2y_4, y_1^3-x_3x_4^2y_3\rangle$.
Indeed, $\sA$ and $\langle B'\rangle$ can not be glued since $k_1\sA\cap k_2\langle B'\rangle=\{(0,0,0)\}$ for any $k_1,k_2\in\N$ because
the third coordinate of any element in $k_1\sA$ is 0, and the only element in $k_2\langle B'\rangle$ with this property is $(0,0,0)$.
We will see in Proposition \ref{propNecCond} that this implies that $\sA$ and $\sB$ can not be glued.
}\end{example}

\section{When can two semigroups in $\N^n$ be glued?}\label{secWhenGluing}

Consider two finite subsets $A=\{\a_1,\ldots,\a_p\}$ and $B=\{\b_1,\ldots,\b_q\}$ in $\N^n$.
As already observed in Corollary \ref{corNondeg}, if $\sA$ and $\sB$ can be glued into a nondegenerate semigroup $\sC$
then $\rk A+\rk B=n+1$ so this condition is necessary but 
it is not sufficient as Example \ref{exTwistedNo} shows.
So along this section we will only consider semigroups $\sA$ and $\sB$ satisfying the following rank conditions:
\begin{equation}\label{eqRankCond}
\rk{A\vert B}=n \quad\hbox{and}\quad \rk A+\rk B=n+1\,.
\end{equation}

\subsection{The gluable lattice point and the level of a binomial}

Consider the $\Q$-vector spaces generated by $A$ and $B$, $v(A)$ and $v(B)$ respectively. The rank conditions (\ref{eqRankCond}) 
imply that $\dim_\Q v(A)\cap v(B)=1$. Thus, there exists a unique (up to sign) element $\u=(u_1,\ldots,u_n)\in\Z^n$ 
with $\gcd(u_1,\ldots,u_n)=1$ that generates the $\Q$-vector space $v(A)\cap v(B)$.  

\begin{definition}{\rm
Given two semigroups $\sA$ and $\sB$ in $\N^n$ satisfying the rank conditions (\ref{eqRankCond}), the unique element $\u=(u_1,\ldots,u_n)\in\Z^n$
with $\gcd(u_1,\ldots,u_n)=1$ that generates the $\Q$-vector space $v(A)\cap v(B)$ and whose first nonzero entry is positive
is called the {\it gluable lattice point} of $A$ and $B$ and denoted by $\u(A,B)$.
}\end{definition}

The gluable lattice point $\u(A,B)$ can be computed from $A$ and $B$ as follows.  Set $s=\rk{A}$ and $t=\rk{B}$. Then, by the rank conditions  (\ref{eqRankCond}),
$s+t=n+1$. After a renumbering of the columns in $A$ and $B$, if necessary, if $\{\a_1,\ldots,\a_s\}$ and $\{\b_1,\ldots,\b_t\}$  are both $\Q$-linearly independent, 
let $\d=(d_1,\ldots,d_{s+t})$ be the vector formed by the $(n+1)$ signed $n\times n$ minors of the $n\times (n+1)$ matrix  $(\a_1\cdots\a_s \b_1\cdots \b_t)$ such that 
$(\a_1\cdots\a_s \b_1\cdots \b_t)\times \d^t$ is zero. 
We can divide the $d_i$'s by any common factors to get them to be relatively prime. 
Then, 
$$
\begin{pmatrix} \a_1& \cdots& \a_s\end{pmatrix}\times \begin{pmatrix} d_1\\ \vdots\\ d_s\end{pmatrix}=-
\begin{pmatrix} \b_1& \cdots& \b_t\end{pmatrix}\times \begin{pmatrix} d_{s+1}\\ \vdots\\ d_{s+t}\end{pmatrix}
$$
is an element of the $\Q$-vector space $v(A)\cap v(B)$, and hence a multiple of $\u(A,B)$. Factoring out the common factors if any, one gets $\u(A,B)$.
Note that, although $\d$ depends on the choice of the $\Q$-basis for the column spaces of $A$ and $B$, we see that $\u(A,B)$ is independent of that choice.

\begin{example}\label{exTwistedLP}{\rm
In Example \ref{exTwistedYes}, the $3\times 4$ integer matrix formed by the first two columns of $A$ (that are
linearly independent) and the first two columns of $B$ (that are also linearly independent) is $\begin{pmatrix}  4&3&3&3\\ 0&1&3&2\\ 0&0&0&1\end{pmatrix}$
and, one gets that $\d=2(3,-6,2,0)$. Thus, $3\a_1-6\a_2=-2\b_1=(-6,-6,0)$ and hence the gluable lattice point of $A$ and $B$ is $\u(A,B)=(1,1,0)$.
In Example  \ref{exTwistedNo}, one gets that $\u(A,B')=(1,2,0)$.
}\end{example}

We now prove a result that will be useful in the proof of Theorem \ref{thmHowToGlue} and that leads to the definition of level of a binomial in Definition \ref{defLevel}.

\begin{lemma}\label{lemCharBinomial}
Given $A=\{\a_1,\ldots,\a_p\}$ and $B=\{\b_1,\ldots,\b_q\}$ in $\N^n$ satisfying the rank conditions (\ref{eqRankCond}), 
let $\u=\u(A,B)\in\Z^n$ be the gluable lattice point of $A$ and $B$.
Take two relatively prime integers $k_1$ and $k_2$, and set $C=k_1A\cup k_2B$.
Then, a binomial $w\in R$ of the form
$w = \prod_{j\in \alpha_+}x_j^{\alpha_j} \prod_{j\in \beta_-} y_j^{-\beta_j}-
           \prod_{j\in \alpha_-}x_j^{-\alpha_j}\prod_{j\in \beta_+}y_j^{\beta_j}$
with $\alpha\in\Z^p$ and $\beta\in\Z^q$ belongs to $I_C$ \iff{} there exists an integer $\ell\in\Z$ such that
$$
B\cdot \beta=k_1 \ell\u
\quad\hbox{and}\quad 
A\cdot \alpha=k_2 \ell\u\,.
$$
\end{lemma}

\begin{proof}
If $w\in I_C$, then $k_2 B\cdot\beta=k_1 A\cdot\alpha$ and since this is a vector in $v(A)\cap v(B)$, there exists $c\in\Q$ such that
 $k_2 B\cdot\beta=k_1 A\cdot\alpha=c\u$. As $B\cdot\beta$ and $A\cdot\alpha$ belong to $\Z^n$, and $\gcd(u_1,\ldots,u_n)=1$, one has that $c\in\Z$
and $k_1$ and $k_2$ divide $c$. Since we have assumed $k_1$ and $k_2$ to be relatively prime, one gets that $k_1k_2$ divides $c$.
Then, there exists $\ell\in\Z$ such that $k_2 B\cdot\beta=k_1 A\cdot\alpha=k_1k_2\ell\u$ and the result follows.

\smallskip
Conversely, if $B\cdot \beta=k_1 \ell\u$ and $A\cdot \alpha=k_2 \ell\u$ for some $\ell\in\Z$, then $k_1A\cdot \alpha=k_2B\cdot \beta$ and hence
$w\in I_C$.
\end{proof}

\begin{definition}\label{defLevel}{\rm
Given a binomial $w$ as in Lemma \ref{lemCharBinomial} that belongs to $I_C$, the integer $\ell\in\N$ is called the level of $w$
and denoted by $\ell(w)$.
}\end{definition}

\begin{example}\label{exTwistedLevel}{\rm
In Example \ref{exTwistedYes} ($k_1=k_2=1$), if one considers the binomial $w=y_1^2-x_1x_4^2$, then $\alpha=(1,0,0,2)$ and $\beta=(2,0,0,0)$,
and $A\cdot \alpha=B\cdot \beta=6\u(A,B)$ so $w$ has level $\ell(w)=6$.
In Example  \ref{exTwistedNo} ($k_1=k_2=1$), for $w_1=y_1^2y_2-x_3x_4^2y_4$ and $w_2=y_1^3-x_3x_4^2y_3$, one has that $\alpha_1=\alpha_2=(0,0,1,2)$, $\beta_1=(1,2,0,-1)$
and $\beta_2=(3,0,-1,0)$. Thus, $A\cdot \alpha_1=B'\cdot \beta_1= 2\u(A,B')$ and $A\cdot \alpha_2=B'\cdot \beta_2= 2\u(A,B')$ and hence
both binomials $w_1$ and $w_2$ have level 2.
}\end{example}

\subsection{Necessary conditions for $\sA$ and $\sB$ to be glued}\label{subsec2ndNecCond}

As already observed, the rank conditions $\rk{A\vert B}=n$ and $\rk A+\rk B=n+1$ are necessary to have that $\sA$ and $\sB$ can be glued
(into a nondegenerate semigroup) but they are not sufficient as Example \ref{exTwistedNo} shows.

\smallskip

It is also clear that if for $C=k_1 A\cup k_2 B$ for some integers $k_1,k_2$, we want to have in $I_C$ a binomial of the
form $\rho=\underbar{x}^\c-\underbar{y}^\d$ with $\c\in\N^p$ and $\d\in\N^q$, then $k_1A\cdot \c=k_2 B\cdot \d$.
This implies that $A\cdot k_1\c=B\cdot k_2 \d$ and hence, the system 
$A\cdot X=B\cdot Y$ must have nontrivial solution $(X,Y)$ such that $X\in\N^p$ and $Y\in\N^q$.
The condition
\begin{center}
the system $A\cdot X=B\cdot Y$ has a nontrivial solution $(X,Y)$ with $X\in\N^p$ and $Y\in\N^q$
\end{center}
is thus also necessary.

One can easily show that this condition is equivalent to the fact that both semigroups $\sA$ and $\sB$ contain a multiple of the gluable lattice point of $A$ and $B$:

\begin{lemma}\label{lemEquivSecCond}
Given $A=\{\a_1,\ldots,\a_p\}$ and $B=\{\b_1,\ldots,\b_q\}$ in $\N^n$ satisfying the rank conditions
$\rk{A\vert B}=n$ and $\rk A+\rk B=n+1$,
let $\u=\u(A,B)\in\Z^n$ be the gluable lattice point of $A$ and $B$.
Then, the following are equivalent:
\begin{enumerate}
\item\label{condSyst}
the system $A\cdot X=B\cdot Y$ has a nontrivial solution $(X,Y)$ such that $X\in\N^p$ and $Y\in\N^q$;
\item\label{condk1k2}
there exist positive integers $a$ and $b$ such that $a\u\in\sA$ and $b\u\in\sB$.
\end{enumerate}
\end{lemma}

\begin{proof}
If (\ref{condSyst}) holds, then
$x_1\a_1+\cdots+x_p\a_p=y_1\b_1+\ldots+y_q\b_q$ belongs to $\sA$ and $\sB$. In particular, it is an element in $v(A)\cap v(B)$ and hence a
multiple of $\u$ and (\ref{condk1k2}) follows.

\smallskip
Conversely, if  (\ref{condk1k2})  holds, then $ab\u$ belongs to $\sA$ and $\sB$, i.e.,
$ab\u=x_1\a_1+\cdots+x_p\a_p=y_1\b_1+\ldots+y_q\b_q$ for some $x_i,y_j\in\N$ and (\ref{condSyst}) follows.
\end{proof}

\begin{remark}{\rm
Note that if the rank conditions (\ref{eqRankCond}) are satisfied, the equivalent conditions in Lemma \ref{lemEquivSecCond} force the gluable lattice point $\u(A,B)$ to have all
its entries nonnegative, i.e., $\u(A,B)\in\N^n$.
}\end{remark}

We summarize the necessary conditions in the following result:

\begin{proposition}\label{propNecCond}
If two semigroups $\sA$ and $\sB$ in $\N^n$ can be glued into a nondegenerate semigroup $\sC$, i.e., $C=C = k_1A\Join k_2B$ for some relatively prime integers $k_1$ and $k_2$,
then:
\begin{enumerate}
\item
$\rk{A\vert B}=n$ and $\rk A+\rk B=n+1$, and
\item
both semigroups $\sA$ and $\sB$ contain a nonzero multiple of $\u(A,B)$, the gluable lattice point of $A$ and $B$.
\end{enumerate}
\end{proposition}

\subsection{A sufficient condition}

The two necessary conditions in Proposition \ref{propNecCond} are not sufficient but one has very close sufficient conditions:

\begin{theorem}\label{thmHowToGlue}
Consider $A=\{\a_1,\ldots,\a_p\}$ and $B=\{\b_1,\ldots,\b_q\}$ in $\N^n$ satisfying the rank conditions 
$\rk{A\vert B}=n$ and $\rk A+\rk B=n+1$,
and let $\u=\u(A,B)\in\Z^n$ be the gluable lattice point of $A$ and $B$.
If there exist two relatively prime integers $k_1,k_2\in\N$ such that  $k_1\u\in\sB$ and $k_2\u\in\sA$, then $C= k_1A\Join k_2B$
and hence, $\sA$ and $\sB$ can be glued.
\end{theorem}

\begin{proof}
Since $k_1\u\in\sB$ and $k_2\u\in\sA$, there exist $d_1,\ldots,d_q,c_1,\ldots,c_p\in\N$ such that 
$$
k_1\u=d_1\b_1+\cdots+d_q\b_q\quad\hbox{and}\quad 
k_2\u=c_1\a_1+\cdots+c_p\a_p\,.
$$
As $k_1$ and $k_2$ are relatively prime, this implies by Lemma \ref{lemCharBinomial} that the binomial $\rho=x_1^{c_1}\cdots x_p^{c_p}-y_1^{d_1}\cdots y_q^{d_q}$ 
belongs to $I_C$ and that it has level 1, i.e., $\ell(\rho)=1$.

\smallskip
Consider now an arbitrary binomial $w\in I_C$ of the form 
$w = \prod_{j\in \alpha_+}x_j^{\alpha_j} \prod_{j\in \beta_-} y_j^{-\beta_j}-
           \prod_{j\in \alpha_-}x_j^{-\alpha_j}\prod_{j\in \beta_+}y_j^{\beta_j}$
with $\alpha\in\Z^p$ and $\beta\in\Z^q$. Then, by Lemma \ref{lemCharBinomial}, 
\begin{center}
$A\cdot \alpha=k_2 \ell(w) \u$ and $B\cdot \beta=k_1 \ell(w) \u$.
\end{center}
Set
$\displaystyle{\theta=\prod_{j=1}^{p}x_j^{\ell(w)c_j} \prod_{j\in \beta_-} y_j^{-\beta_j}-\prod_{j\in \beta_+}y_j^{\beta_j}
}$. Since $A\cdot (\ell(w)c_1,\ldots,\ell(w)c_p)^t=\ell(w)(c_1\a_1+\cdots+c_p\a_p)=\ell(w)k_2\u$ and
$B\cdot \beta=k_1 \ell(w) \u$, one has by Lemma \ref{lemCharBinomial} that $\theta$ is a binimial in $I_C$ of level
$\ell(\theta)=\ell(w)$.
On the other hand, the binonial
$\rho^{[\ell(w)]}:=\prod_{j=1}^{p}x_j^{\ell(w)c_j} -\prod_{j=1}^{q}y_j^{\ell(w)d_j} $ also belongs to $I_C$ since it is a multiple
of $\rho$. Indeed, it also has level $\ell(\rho^{[\ell(w)]})=\ell(w)$.
Now, 
$\theta-\rho^{[\ell(w)]}\cdot \prod_{j\in \beta_-} y_j^{-\beta_j}=
\prod_{j=1}^{q}y_j^{\ell(w)d_j}\prod_{j\in \beta_-} y_j^{-\beta_j}
-\prod_{j\in \beta_+}y_j^{\beta_j}$
that belongs to $I_C$ and is, again, of level $\ell(w)$. Moreover, this binomial only involves variables $y_1,\ldots,y_q$ so it is in
$I_C\cap k[y_1,\ldots,y_q]=I_B$. This shows that $\theta\in I_A\cdot R+\langle \rho \rangle$.

\smallskip
Finally, consider now
$w-\theta\cdot  \prod_{j\in \alpha_-}x_j^{-\alpha_j}$:
\begin{eqnarray*}
w-\theta\cdot  \prod_{j\in \alpha_-}x_j^{-\alpha_j}
&=&
\prod_{j\in \alpha_+}x_j^{\alpha_j} \prod_{j\in \beta_-} y_j^{-\beta_j}-
\prod_{j=1}^{p}x_j^{\ell(w)c_j} \prod_{j\in \beta_-} y_j^{-\beta_j}\prod_{j\in \alpha_-}x_j^{-\alpha_j}\\
&=&
 \prod_{j\in \beta_-} y_j^{-\beta_j}(
\prod_{j\in \alpha_+}x_j^{\alpha_j}-
\prod_{j=1}^{p}x_j^{\ell(w)c_j}\prod_{j\in \alpha_-}x_j^{-\alpha_j})\,.
\end{eqnarray*}
Again, one checks that the binomial $\displaystyle{\prod_{j\in \alpha_+}x_j^{\alpha_j}-\prod_{j=1}^{p}x_j^{\ell(w)c_j}\prod_{j\in \alpha_-}x_j^{-\alpha_j}}$
satisfies the conditions in Lemma \ref{lemCharBinomial}: it belongs to $I_C$ and has level $\ell(w)$. Since it involves only variables
$x_1,\ldots,x_p$, it is in $I_A$ and we have proved that $w-\theta\cdot  \prod_{j\in \alpha_-}x_j^{-\alpha_j}\in I_A\cdot R$.
Putting all together, one gets that $w\in I_A\cdot R+I_B\cdot R+\langle\rho\rangle$. Since the other inclusion holds, we have shown that
$I_C=I_A\cdot R+I_B\cdot R+\langle\rho\rangle$ and hence, $\sC$ is a gluing of $\sA$ and $\sB$.
\end{proof}

\begin{remark}{\rm
Note that the conditions on $k_1$ and $k_2$ in Theorem \ref{thmHowToGlue},
$k_1\u\in\sB$ and $k_2\u\in\sA$, are sufficient but not necessary as Example \ref{exTwistedYes} shows. In this example, 
we saw later in Example \ref{exTwistedLP} that $\u=\u(A,B)=(1,1,0)$, and
Theorem \ref{thmHowToGlue} says that for $C=3A\cup 2B$, one has that $\sC$ is a gluing of $\sA$ and $\sB$ which is true. But, as observed
in Example \ref{exTwistedYes}, for  $C=A\cup B$ one also has that $\sC$ is a gluing of $\sA$ and $\sB$ while $\u\notin\sA$ and $\u\notin\sB$.
}\end{remark}

\begin{example}\label{exPrimeGluing}{\rm
In $\N^3$, consider the sets $A=\{\a_1,\ldots,\a_4\}$, $B=\{\b_1,\ldots,\b_4\}$ and $C=A\cup B$ with
$$
\a_1=(1,6,7), \a_2=(1,4,5),  \a_3=(1,2,3), \a_4=(2,2,4),
$$
$$
\b_1=(1,1,6), \b_2=(1,1,4), \b_3=(1,1,1), \b_4=(3,3,6).
$$
The matrices $A$ and $B$ have rank 2 and the matrix $C$ has rank 3. The defining ideals of the semigroup rings are
$$
I_A=\langle x_3^3-x_2x_4, x_2^2-x_1x_3, x_2x_3^2-x_1x_4\rangle,\ 
I_B=\langle y_2y_3^2-y_4, y_2^5-y_1^3y_3^2\rangle,
$$
$$
I_C=\langle  
y_2y_3^2-y_4, x_4y_2-y_1y_3^2, x_3^3-x_2x_4, x_2^2-x_1x_3, x_2x_3^2-x_1x_4, x_4y_1^2-y_2^4,
x_4^2y_1-y_2^2y_4, x_4^3-y_4^2
\rangle 
$$
and these are minimal generating sets (all ideals are homogeneous if one gives to the variables $x_1,\ldots,x_4, y_1,\ldots,y_4$ weights 
$14,10,6,8,8,6,3,12$). One has that
$$
I_C=I_A+I_B+\langle x_4^3-y_4^2 \rangle+ \langle x_4y_2-y_1y_3^2, x_4y_1^2-y_2^4, x_4^2y_1-y_2^2y_4\rangle
$$
and $\sC$ is not a gluing of $\sA$ and $\sB$.
But $v(A)=\{(x,y,z)\in\Q^3\,/\ z=x+y\}$ and $v(B)=\{(x,y,z)\in\Q^3\,/\ x=y\}$ so $\u(A,B)=(1,1,2)$ so for $C=3A\cup 2B$, one has
that $\sC$ is a gluing of $\sA$ and $\sB$ by Theorem \ref{thmHowToGlue}. One can check using \cite{Sing} that
$I_C=I_A+I_B+\langle x_4^3-y_4^2\rangle$.
}\end{example}

The following example shows that one has to assume in Theorem \ref{thmHowToGlue}  that $k_1$ and $k_2$ are relatively prime,
i.e., that the necessary conditions in Proposition \ref{propNecCond}
are not sufficient to glue two semigroups.

\begin{example}\label{exNotPrimeNotGluing}{\rm
Consider that subsets $A=\{\a_1,\a_2,\a_3,\a_4\}$ and $B=\{\b_1,\b_2,\b_3,\b_4\}$ in $\N^3$ whose elements are the columns of the matrices
$$
A=\begin{pmatrix}1&1&2&3\\ 6&4&5&3\\ 7&5&7&6\end{pmatrix}\quad,\quad
B=\begin{pmatrix}1&2&3&9\\ 1&2&3&9\\  6&7&8&18\end{pmatrix}
$$
In this example, $v(A)$ and $v(B)$ are the same as in Example  \ref{exPrimeGluing} so the necessary conditions in Proposition \ref{propNecCond}
are satisfied,
and $\u=(1,1,2)$. The difference here is that no multiple of $\u$ belongs to $\langle \a_1,\a_2,\a_3 \rangle$ nor  $\langle \b_1,\b_2,\b_3 \rangle$. 
So the only multiples of $\u$ in $\sA$, respectively $\sB$, are the multiples of $\a_4=3\u$, respectively $\b_4=9\u$ and
one can not find $k_1$ and $k_2$ relatively prime satisfying the hypothesis in Theorem \ref{thmHowToGlue}.
Indeed, we could not find in this example $k_1$ and $k_2$ such that $ k_1A\Join k_2B$.
}\end{example}

But the condition in Theorem \ref{thmHowToGlue} is not an characterization. 
It may occur that we can glue two semigroups $\sA$ and $\sB$ even if 
there does not exist relatively prime integers $k_1$ and $k_2$ such that
 $k_1\u(A,B)\in\sB$ and $k_2\u(A,B)\in\sA$
as the following example shows.

\begin{example}\label{exNotPrimeYesGluing}{\rm
In Example \ref{exNotPrimeNotGluing}, if do not change the first three generators of $\sA$ and $\sB$ but now consider
$\a_4=(5,5,10)$ and $\b_4=(10,10,20)$, we are exactly in the same situation so Theorem \ref{thmHowToGlue} does not apply.
But here one can check using \cite{Sing} that for $C=2A\cup B$, $\sC$ is a gluing of $\sA$ and $\sB$: 
$I_C=I_A+I_B+\langle x_4-y_4\rangle$. So in this example, $\sA$ and $\sB$ can be glued.
}\end{example}

We will summarize the results into this theorem.  

\begin{theorem}\label {thmSummary} 
Let $A$ and $B$ be two finite sets in $\N^n$ satisfying the  rank conditions 
$$
\rk{A\vert B}=n \quad\hbox{and}\quad \rk A+\rk B=n+1
$$
and let $\u=\u(A,B)\in\Z^n$ be the gluable lattice point of $A$ and $B$.
Then, 
\begin{center}
(a) $\Longrightarrow$ (b) $\Longrightarrow$ (c) $\Longleftrightarrow$ (d) 
\end{center}
for the following 4 conditions:
\begin{enumerate}
\item[(a)]
there exist relatively prime positive integers $k_1,k_2 $ such that $k_2\u\in\sA$ and $k_1\u\in\sB$;
\item[(b)]
$\sA$ and $\sB$ can be glued;
\item[(c)]
there exists positive integers $k_1,k_2$ such that $k_2\u\in\sA$ and $k_1\u\in\sB$;
\item[(d)]
the  system $A\cdot X=B\cdot Y$ has a nontrivial solution $(X,Y)$ with $X\in\N^p$ and $Y\in\N^q$.
\end{enumerate}
\end{theorem}

\begin{proof}
(a) $\implies$ (b) is Theorem \ref{thmHowToGlue},
(b) $\implies$ (d) as observed at the beginning of section \ref{subsec2ndNecCond}, and
(c) $\Longleftrightarrow$ (d) is Lemma \ref{lemEquivSecCond}. 
\end{proof}

Iin the previous result, (b) $\not\Rightarrow$ (a) by Example  \ref{exNotPrimeYesGluing}
and (c) $\not\Rightarrow$ (b) by Example  \ref{exNotPrimeNotGluing}.

\section{Consequences and examples}\label{secConsequences}

\subsection{Numerical semigroups}

Two numerical semigroups can always be glued. One recovers this well-known fact from Theorem \ref{thmSummary}.
When $n=1$, the rank conditions (\ref{eqRankCond}) are always satisfied since $\rk{A}=\rk{B}=\rk{A|B}=1$ for any finite subsets $A$ and $B$ of $\N$ defining numerical semigroups.
Moreover, $\u(A,B)=1$ for all $A$ and $B$. Thus, one can always take an integer $k_1$ in $\sB$ and an integer $k_2$ in $\sA$ and get
condition (a) in Theorem \ref{thmSummary} satisfied.

\subsection{Rank 1 semigroups}

The case when $\rk{B}=1$ was studied in \cite{sforum20}. This occurs if either $B=\{\b\}$, i.e., $B$ has only one element ($q=1$), 
or there exists $\b\in\N^n$ such that $B=\b\times (u_1\cdots u_q)$ for $u_1,\ldots,u_q\in\N$, i.e., all the elements of $B$ are multiples of a same 
element $\b\in\N^n$.
This is the most degenerate case and, as observed in Corollary \ref{corNondeg}, if $\sA$ is nondegenerate, then $\rk{B}=1$ is necessary for
$\sA$ and $\sB$ to be glued. 

\begin{theorem}[{\cite[Thms. 3 and 5]{sforum20}}]
If $\rk{A}=n$ and $\rk{B}=1$, i.e., either $B=\{\b\}$, or there exists $\b\in\N^n$ such that $B=\b\times (u_1\cdots u_q)$ for $u_1,\ldots,u_q\in\N$, then $\sA$ and $\sB$ can be glued if and only if a multiple of $\b$ belongs to $\sA$.
\end{theorem}

\begin{proof}
In this case, $\u(A,B)=\b$ and conditions
(a) and (c) in Theorem \ref{thmSummary} are both equivalent to having a multiple of $\b$ in $\sA$, and the above characterization
follows. 
\end{proof}

\subsection{A case study: Gluing nondegenerate homogeneous semigroups of $\N^2$ in $\N^3$}\label{secHowToDim2}

Given two sequences of integers $0<a_1\ldots < a_{p-2} <c$ and $0<b_1\ldots < b_{q-2} <d$, one can consider
the nondegenerate semigroups $\sA$ and $\sB$ in $\N^2$ generated by the columns of the matrices
$$
A = \begin{pmatrix} 
c & c-a_1& \ldots& c-a_{p-2} & 0\\
0& a_1 &\ldots &a_{p-2}&c\\
\end{pmatrix}\quad\hbox{and}\quad
B= \begin{pmatrix} 
d & d-b_1& \ldots& d-b_{q-2} & 0\\
0& b_1 &\ldots &b_{q-2}&d\\
\end{pmatrix}\,.
$$
Since the columns of $A$ and $B$ add up to the same number,  $c$ and $d$ respectively,
the semigroup rings $k[A]$ and $k[B]$ are graded with the standard grading,
i.e., the ideals $I_A$ and $I_B$ are homogeneous.
Indeed, $I_A$ and $I_B$ are the defining ideals of two projective monomial curves $\C_A\subset\PP^{p-1}$ and $\C_B\subset\PP^{q-1}$
of degree $c$ and $d$ respectively.
We will say that the semigroups $\sA$ and $\sB$ have degree $c$ and $d$.

\smallskip

We know that $\sA$ and $\sB$ can not be glued in $\N^2$ because they are nondegenerate (Theorem \ref{thmDeg}). We will
now show how we can glue them by embedding them in $\N^3$ and gluing them in $\N^3$ where they are both degenerate. 

\smallskip

Suppose that $d$ is relatively prime to $a_i$ for some $i, 1\le i \le p-2$.  
Dividing $c$ by $a_i$, we write  $c= m a_i -r_i$ for some $m\geq 2$ and $0\le r_i < a_i$.  

\smallskip

One can embed $\sA$ and $\sB$ in $\N^3$ by adding to the matrices $A$ and $B$, for example, a zero row.
Moreover, since the ideals $I_A$ and $I_B$ are homogeneous, adding any integer to a row of the matrices will not change the ideal.
So if one considers the following two matrices
\begin{eqnarray*}
A' &=& \begin{pmatrix} 
c+r_i & c-a_1+r_i& \ldots & c-a_{p-2}+r_i & r_i\\
0& a_1 &\ldots &a_{p-2}&c\\
0&0&\ldots&0&0\\
\end{pmatrix}\ \hbox{and}\\
B' &=& \begin{pmatrix} (m-1)d &(m-1)d &\ldots &(m-1)d&(m-1)d\\
d& d-b_1& \ldots& d-b_{q-2}&0\\
0& b_1 &\ldots &b_{q-2}&d\\
\end{pmatrix}
\end{eqnarray*}
one has that $I_A=I_{A'}$ and $I_B=I_{B'}$.
Moreover, $\rk{A'}=\rk{A}=2$, $\rk{B'}=\rk{B}=2$ and $\rk{A'|B'}=3$, and hence the rank conditions (\ref{eqRankCond}) are satisfied. 
Now focusing of the $(i+1)$-th column of $A'$ and the first column of $B'$, we have that 
$\u = \begin{pmatrix}
m-1\\
1\\
0\\
\end{pmatrix}$  
belongs to the intersection of the $\Q$-vector spaces generated by the columns of $A'$ and $B'$, i.e., it is the gluable lattice point of $A'$ and $B'$.
Further, $a_i\u \in \sA$ and $d \u \in \sB$.  Since $(d,a_i) =1$ by $(a)\implies (b)$ of Theorem \ref{thmSummary}, we see that 
$\langle A'\rangle$ and $\langle B'\rangle$ can be glued.  
By Theorem \ref{thmHowToGlue}, $C=dA' \Join a_iB'$ is a gluing of $A'$ and $B'$ and $I_C=I_A+I_B+\langle\rho\rangle$
for $\rho=x_{i+1}-y_1$. Note that $\sC$ is homogeneous of degree $mda_i$.

\smallskip

Thus the two homogeneous nondegenerate semigroups in dimension 2 can be glued after an appropriate embedding in dimension 3. 
This is certainly not the only possible embeddings that can be glued. 

\begin{example}{\rm
If we choose for $\sA$ and $\sB$ two copies of the semigroup defining the twisted cubic, i.e.,
$
A=B=\begin{pmatrix} 3&2&1&0 \\ 0&1&2&3\end{pmatrix}
$, then for $i=2$, one has that $a_i=2$ and hence $m=2$ and $r_i=1$. Thus, for
$$
C=\begin{pmatrix} 12&9&6&3&6&6&6&6 \\ 0&3&6&9&6&4&2&0\\ 0&0&0&0&0&2&4&6\end{pmatrix}
$$
one has that $C=3A'\Join 2B'$ and $I_C=I_A+I_B+\langle x_3-y_1\rangle$, where $I_A\subset k[x_1,\ldots,x_4]$ and $I_B\subset k[y_1,\ldots,y_4]$ are
two copies of the defining ideal of the twisted cubic.
}\end{example}

Inspired by the situation in the previous example, we say that $\sC$ is a {\it selfgluing} of $\sA$ when one can embed a nondegenerate semigroup $\sA\subset\N^n$ in $\N^{2n-1}$
in two different ways, $\langle A'\rangle\subset\N^{2n-1}$ and $\langle A''\rangle\subset\N^{2n-1}$ with $\rk{A'}=\rk{A''}=\rk{A}=n$ and $I_{A'}=I_{A''}=I_A$,
and such that there exist two relatively prime integers $k_1,k_2$ such that
$C=k_1A'\Join k_2A''$. Then, $I_C=I_A+I_B+\langle\rho\rangle$ where $I_A\subset k[\xx]$ and $I_B\subset k[\yy]$ are two copies of the same ideal, the defining ideal of $\rA$, and
$\rho$ is a gluing binomial.
This will be developed in a forthcoming paper.

\subsection{Final illustrating examples}

We finish this paper with a series of example illustrating some interesting facts.

\subsubsection{The number of minimal generators}

If $C=k_1A\Join k_2 B$ then, by definition, the number of elements in a minimal generating set of $I_A$, $I_B$ and $I_C$ (recall that these three ideals
are graded if one gives the correct weights to the variables) are related by $\mu(I_C)=\mu(I_A)+\mu(I_B)+1$.
One has to be careful with the number of minimal generators in $I_C$ since one could have that $\mu(I_C)=\mu(I_A)+\mu(I_B)+1$ with
$\sC$ that is not a gluing of $\sA$ and $\sB$ as the following example shows.

\begin{example}{\rm
For $A=\{(1,6,7), (1,4,5), (2,5,7), (5,5,10)\}$ and $B=\{(1,1,6), (2,2,7),$ $(3,3,8), (10,10,20)\}$, one has that
$I_A$ is minimally generated by 3 binomials, $I_A=\langle  x_2^2x_3^3-x_1^3x_4, x_3^5-x_2^5x_4, x_2^7-x_1^3x_3^2\rangle$,
$I_B$ is minimally generated by 2 elements (it is a complete intersection),
$I_B=\langle y_2^2-y_1y_3, y_3^4-y_1^2y_4 \rangle$ and $I_C$ is minimally generated by 6 elements but it is not a gluing.
One has that $I_C=I_A+I_B+\langle x_4^2-y_4 \rangle+\langle x_4y_1-y_3^2 \rangle$. The thing is that the second minimal generator of
$I_B$ is not minimal in $I_C$.
}\end{example} 

\subsubsection{Linear binomial}

The existence of a binomial of the form $x_i-y_j$ in the ideal $I_C$ for $C=k_1\sqcup k_2B$ does not mean that $C=k_1A\Join k_2B$ as the following example shows.

\begin{example}{\rm
In $\N^3$, consider
$A=\{(4,1,5), (2,1,3), (1,2,3), (3,1,4), (3,3,6)\}$ and 
$B=\{(2,1,1), (3,2,3), (5,3,4), (4,5,11)), (3,3,6)\}$.
One has that $\rk{A}=\rk{B}=2$ and that the gluable lattice point of $A$ and $B$ is $\u=(1,1,2)$.
Using Singular \cite{Sing}, one can check that for $C=A\sqcup B$, $I_C$ is minimally generated by 13 binomials while $I_A$ and $I_B$ are minimally generated
by 4 and 3 binomials respectively, and hence $\sC$ is not a gluing of $\sA$ and $\sB$,
even if $x_5-y_5$ is one of the 13 minimal generators of $I_C$.
}\end{example} 

\subsubsection{Cohen-Macaulay and non Cohen-Macaulay homogeneous surfaces}

Using the way to embed nondegenerate semigroups in higher dimensional spaces where they are degenerate and glue them there as mentioned in Section \ref{secHowToDim2}
together with Theorem \ref{thmCM}, one can easily build Cohen-Macaulay or non Cohen-Macaulay surfaces.
Let's give two examples, the first one obtained by selfgluing a Cohen-Macaulay curve in $\PP^3$ with itself to get a Cohen-Macaulay surface in $\PP^7$, and the other by selfgluing 
a non Cohen-Macaulay curve in $\PP^3$ with itself to get a non Cohen-Macaulay surface in $\PP^7$.

\begin{example}{\rm
The monomial curve in $\PP^3$ defined by the semigroup $\sA$ generated by the columns of the matrix
$A=\begin{pmatrix} 5&4&3&0\\ 0&1&2&5\end{pmatrix}$ is known to be Cohen-Macaulay. It is indeed of Hilbert-Burch and 
$I_A=\langle x_1x_3-x_2^2, x_1x_2x_4-x_3^3, x_1^2x_4-x_2x_3^2\rangle$.
Then, for $C=\begin{pmatrix} 25&20&15&0&20&20&20&20\\ 0&5&10&25&5&4&3&0\\ 0&0&0&0&0&1&2&5\end{pmatrix}$,
the semigroup generated by the columns of $C$ defines a projective surface  in $\PP^7$ that is Cohen-Macaulay by Theorem \ref{thmCM}.
}\end{example}

\begin{example}{\rm
The monomial curve in $\PP^3$ defined by the semigroup $\sA$ generated by the columns of the matrix
$A=\begin{pmatrix} 5&4&1&0\\ 0&1&4&5\end{pmatrix}$ is known to be non Cohen-Macaulay and $I_A$ is minimally generated by 5 elements,
$I_A=\langle x_1x_4-x_2x_3, x_2x_4^3-x_3^4, x_1x_3^3-x_2^2x_4^2, x_1^2x_3^2-x_2^3x_4, x_1^3x_3-x_2^4\rangle$.
Then, for $C=\begin{pmatrix} 25&20&5&0&20&20&20&20\\ 0&5&20&25&5&4&1&0\\ 0&0&0&0&0&1&4&5\end{pmatrix}$,
the semigroup generated by the columns of $C$ defines a projective surface in $\PP^7$ that is not Cohen-Macaulay by Theorem \ref{thmCM}.
}\end{example}

\end{document}